\documentclass[12pt,final]{amsart}
\usepackage{amssymb, amsmath, amsthm}
\usepackage{mathrsfs}  
\usepackage{graphicx}
\usepackage[colorlinks=true, citecolor=blue, linkcolor=blue, urlcolor=blue]{hyperref}
\usepackage[text={6in,8in},centering]{geometry}
 \usepackage{xcolor}

\usepackage{mathrsfs}  

\usepackage{bm}

\usepackage{ifdraft}
\ifoptionfinal{
\usepackage[disable]{todonotes}
}{
\usepackage[norefs, nocites]{refcheck}

\usepackage[notref, notcite]{showkeys}
\usepackage[bordercolor=white, color=white]{todonotes}
}
\newcommand{\HOX}[1]{\todo[noline, size=\footnotesize]{#1}}

\makeatletter\providecommand\@dotsep{5}\def\listtodoname{List of Todos}\def\listoftodos{\hypersetup{linkcolor=black}\@starttoc{tdo}\listtodoname\hypersetup{linkcolor=blue}}\makeatother

\usepackage{hyperref}
\usepackage{etoolbox}

\usepackage[colored]{shadethm} 
\definecolor{shaderulecolor}{rgb}{0.651,0.074,0.090}

\newshadetheorem{theorem}{Theorem}[section]
\newtheorem{lemma}[theorem]{Lemma}
\newtheorem{proposition}[theorem]{Proposition}
\newtheorem{corollary}[theorem]{Corollary}
\theoremstyle{definition}
\newtheorem{definition}[theorem]{Definition}
\theoremstyle{remark}
\newtheorem{remark}[theorem]{Remark}

\def\R{\mathbb R}

\DeclareMathOperator{\ND}{\text{\Large$\triangleleft$}}
\DeclareMathOperator{\NS}{\underline{\text{\Large$\triangleleft$}}}

\renewcommand{\leq}{\leqslant}
\renewcommand{\geq}{\geqslant}

\def\p{\partial}
\DeclareMathOperator{\supp}{supp}
\newcommand{\pair}[1]{g(#1)}

\makeatletter
\newcommand*\KN{\mathpalette\@KN\relax}
\newcommand*\@KN[1]{%
  \mathbin{%
    \ooalign{%
      $#1\m@th\bigcirc$\cr
      \hidewidth$#1\m@th\wedge$\hidewidth\cr
    }%
  }%
}
\makeatother

\DeclareMathOperator{\Hess}{Hess}

\newcommand*\xbar[1]{%
   \hbox{%
     \vbox{%
       \hrule height 0.5pt 
       \kern0.5ex
       \hbox{%
         \ensuremath{#1}%
       }%
     }%
   }%
} 

\title[Lorentzian Calder\'{o}n problem near the Minkowski geometry]{Lorentzian Calder\'{o}n Problem near the Minkowski geometry}
\author[S. Alexakis]{Spyros Alexakis}
\address{Department of Mathematics, University of Toronto, 40 St George St, Toronto, ON, Canada M5S 2E4.}
\email{alexakis@math.toronto.edu}
\author[A. Feizmohammadi]{Ali Feizmohammadi}
\address{The Fields institute, 222 College St, Toronto, ON, Canada M5T 3J1.}
\email{afeizmoh@fields.utoronto.ca}
\author[L. Oksanen]{Lauri Oksanen}
\address{Department of Mathematics and Statistics, University of Helsinki, P.O 68, 00014, University of Helsinki, Finland.}
\email{lauri.oksanen@helsinki.fi}


\begin{document}


\maketitle
\begin{abstract}
We study a Lorentzian version of the well-known Calder\'{o}n problem that is concerned with determination of lower order coefficients in a wave equation on a smooth Lorentzian manifold, given the associated Dirichlet-to-Neumann map. In the earlier work of the authors it was shown that zeroth order coefficients can be uniquely determined under a two-sided spacetime curvature bound and the additional assumption that there are no conjugate points along null or spacelike geodesics. In this paper we show that uniqueness for the zeroth order coefficient holds for manifolds satisfying a weaker curvature bound as well as spacetime perturbations of such manifolds. This relies on a new optimal unique continuation principle for the wave equation in the exterior regions of double null cones. In particular, we solve the Lorentzian Calder\'{o}n problem near the Minkowski geometry. 
\end{abstract}

\section{Introduction}
\subsection{Formulation of the problem}
We start with the geometric setup and let $(M,g)$ be a Lorentzian manifold of dimension $1+n$ with signature $(-,+,\ldots,+)$. We make the standing assumption that the manifold is of the form
\begin{equation}\label{cylinder}
M=[-T,T]\times M_0
\end{equation}
for some $T > 0$ and a compact connected manifold $M_0$ with a smooth boundary. The metric $g$ is assumed to be of the form
\begin{equation}\label{splitting} g(t,x)= c(t,x)\left(-dt^2+ g_0(t,x)\right),\quad \forall\, (t,x)\in M,\end{equation}
where $c$ is a smooth strictly positive function on $M$ and $g_0(t,\cdot)$ is a family of smooth Riemannian metrics on $M_0$ that depend smoothly on the variable $t\in [-T,T]$.
We remark that if a Lorentzian manifold with timelike boundary admits a global time function then it is isometric to a manifold of the form \eqref{cylinder}--\eqref{splitting}, see \cite[Appendix A]{AFO} for the details.

Let $V \in C^{\infty}(M)$. We consider the following wave equation on $(M,g)$, with the zeroth order coefficient given by $V$,
\begin{equation}\label{pf0}
\begin{aligned}
\begin{cases}
(\Box+V)u=0\,\quad &\text{on $M$},
\\
u=f\,\quad &\text{on $\Sigma=(-T,T)\times \p M_0$,}\\
u(-T,x)=\p_{t}u(-T,x)=0 \,\quad &\text{on $M_0$.}
\end{cases}
    \end{aligned}
\end{equation}
Here, the wave operator $\Box$ is given in local coordinates $(t=x^0,\ldots,x^n)$ on $M$ by
$$\Box u= -\sum_{j,k=0}^n \left|\det g \right|^{-\frac{1}{2}}\frac{\p}{\p x^j}\left(\left|\det g \right|^{\frac{1}{2}}g^{jk}\frac{\p }{\p x^k}\,u\right).$$
It is classical (see for example \cite[Theorem 4.1]{LLT}) that given any $f \in H^1_0(\Sigma)$, equation \eqref{pf0} admits a unique solution $u$ in the energy space
\begin{equation}
\label{natural_space}
C(-T,T;H^1(M_0))\cap C^1(-T,T;L^2(M_0)).
\end{equation}
Moreover, $\p_\nu u|_{\Sigma} \in L^2(\Sigma)$ where $\nu$ is the outward unit normal vector field on $\Sigma$. 

We define the Dirichlet-to-Neumann map, $\Lambda_{V}: H^1_0(\Sigma) \to L^2(\Sigma),$ by
\begin{equation}
\label{DNmap}
\Lambda_{V} f= \p_\nu u\big|_{\Sigma},
\end{equation}
where $u$ is the unique solution to \eqref{pf0} subject to the boundary value $f$ on $\Sigma$.

We are interested in the inverse problem of determining the coefficient $V$ from the knowledge of the Dirichlet-to-Neumann map $\Lambda_{V}$, or in other words, the question of injectivity of the map $V \mapsto \Lambda_{V}$. Following \cite{AFO}, we call this the Lorentzian Calder\'{o}n problem.

\subsection{Obstruction to uniqueness}
There is a natural obstruction to uniqueness for the coefficient $V$ that is due to finite speed of propagation for the wave equation. In order to discuss this obstruction, let us first fix some notation. We write $p \le q$ for points $p,q \in M$
if there is a piecewise smooth causal path on $M$ from $p$ to $q$ or $p=q$. Also, we write $p\ll q$ if there exists a future pointing piecewise smooth timelike
path on $M$ from $p$ to $q$. Using these relations, the causal future and past of
a point $p \in M$ is defined via
\begin{equation}
\label{causal_rel}
J^+(p) = \{q \in M \,:\, p \le q\} \quad \text{and}\quad J^-(p)= \{q \in M\,:\, q \le p\}.
\end{equation}
We also write $I^{\pm}(p)$ for the chronological future and past of the point $p$ that are defined as in \eqref{causal_rel} with $\le$ replaced by $\ll$. 

Using the above notation, we note that by finite speed of propagation for the wave equation, the map $\Lambda_{V}$ contains no information about the coefficient $V$ on the subset 
$$ \mathcal D= \{p \in M\,: J^+(p)\cap \Sigma =\emptyset \text{ or } J^-(p)\cap \Sigma =\emptyset\}.$$
We refer the reader to \cite[Section 1.1]{KO19} for the details. In order to remove this obstruction, we will assume throughout this paper that the time interval $T$ is large in comparison to the support of the coefficient $V$. 

\subsection{Main results}
Without any further assumptions on $(M,g)$, the Lorentzian Calder\'{o}n problem is wide open. We will solve the Lorentzian Calder\'{o}n problem for Lorentzian manifolds that satisfy certain abstract geometric assumptions. We also solve the Lorentzian Calder\'{o}n problem for $C^2(M)$-perturbations of such manifolds. Our abstract geometric assumptions are roughly of the following three types:
\begin{itemize}
\item{A curvature bound on the manifold.}
\item{An assumption that is reminiscent to the notion of simplicity in Riemannian geometry.}
\item{An assumption on the size of the final time $T$ compared to the support of the unknown coefficient $V$.}
\end{itemize}
We begin with the curvature bound. We assume that
\begin{enumerate}
\item[(H1)] {For any point $p\in M$, any spacelike vector $v\in T_pM$, and any null vector $N\in T_pM$ with $g(v,N)=0$, there holds $$g(R(N,v)v,N)\leq 0,$$}
\end{enumerate}
where $R$ stands for the curvature tensor on $(M,g)$. It can be readily verified that the above curvature condition is weaker than the curvature assumption imposed in our earlier work \cite{AFO}. Next, we discuss the simplicity assumption that was also used in our earlier work. Given any $p \in M$, we define the set 
\begin{equation}
\label{exterior_null}
\mathscr E_p = M \setminus (J^-(p)\cup J^+(p)),
\end{equation}
and call it the exterior of the double null cone emanating from the point $p$, see Figure~\ref{default}. We assume that
\begin{figure}
\includegraphics[trim={5cm 2cm 3cm 0},clip,width=0.3 \textwidth]{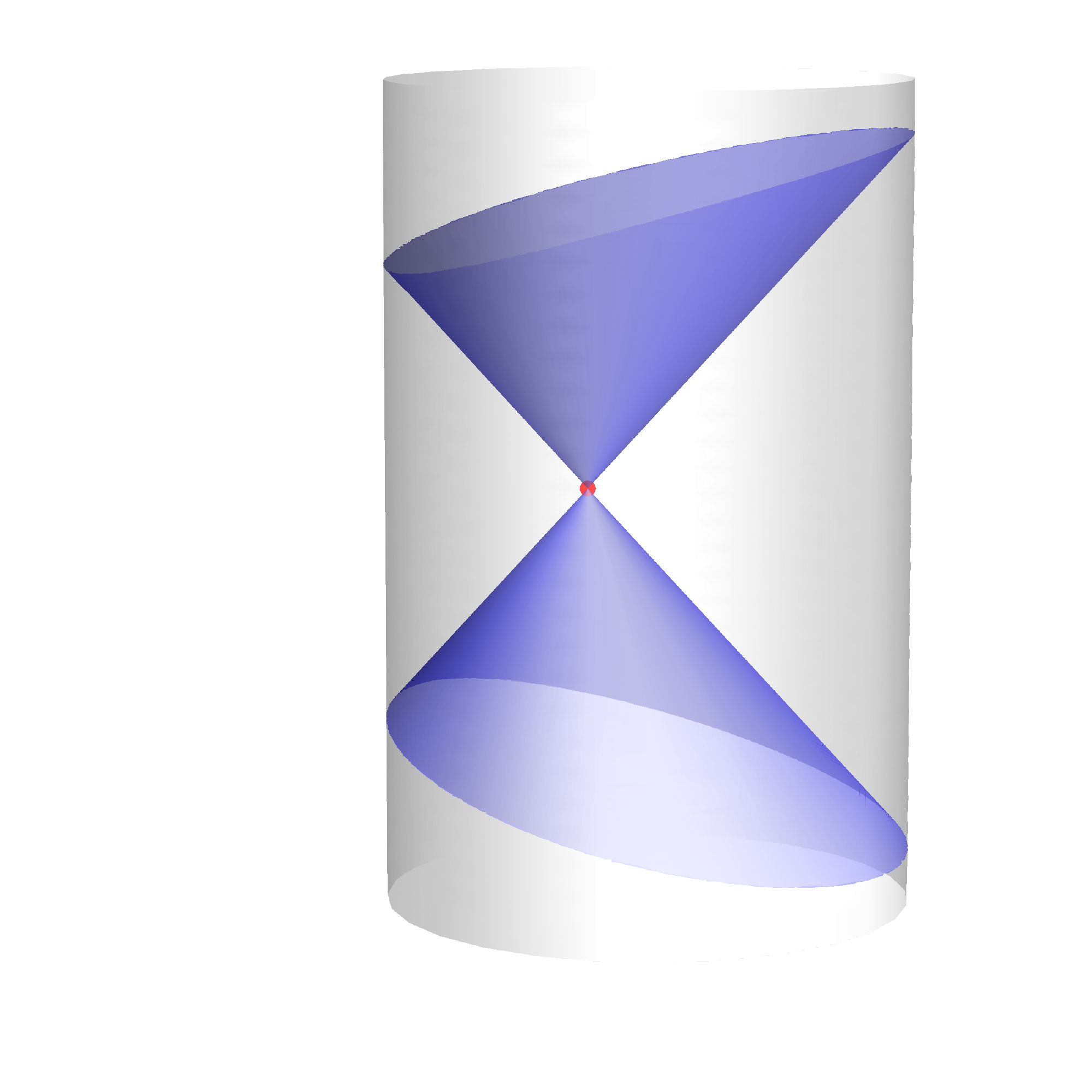}{\centering}
\caption{The schematic for the exterior of the double null cone in the setting of Minkowski geometry in $\R^{1+2}$. The point $p$ is shown in red, $\Sigma$ is gray and $\p \mathscr E_p \cap  M^{\textrm{int}}$ is shown in blue.}
    \label{default}
\end{figure}
\begin{enumerate}
\item[(H2)] For any null geodesic $\gamma$ and any two points $p,q$ on $\gamma$, the only causal path between $p$ and $q$ is along $\gamma$.
For all $p \in M$, the exponential map $\exp_p$
is a diffeomorphism from the spacelike vectors onto $\mathscr E_p$.
\end{enumerate}

In addition to (H1) and (H2) we need to make an assumption that removes the possibility of {\em trapped} null geodesics and imposes a condition on the size of the final time $T$. Precisely, we assume that
\begin{enumerate}
\item[(H3)] Given any $p \in  \supp V$, there holds $\mathscr E_p \cap \p M\subset \Sigma$. Moreover, there is $p_0 \in M^{\textrm{int}}\cap J^-(\supp V)$ such that 
$$\mathscr E_{p_0} \cap \p M \subset \Sigma\quad \text{and}\quad\mathscr E_{p_0} \cap \mathscr E_q =\emptyset,$$
for any $q \in \supp V$.
\end{enumerate}
The significance of the point $p_0$ is related to existence of a suitable exact controllability theory on time slices that are in the causal future of $p_0$. The reader should interpret $p_0$ as a fixed point that is in sufficiently distant past of the region where we are interested in recovering the unknown coefficient $V$. 

Finally, we assume that 
\begin{enumerate}
\item[(H4)] All null geodesics have at most first order of contact with the boundary.
\end{enumerate}
This minor technical condition is also related to the well-known geometric conditions of exact controllability of the wave equation on $M$, see \cite{BG,BLR}. We have the following uniqueness result for the Lorentzian Calder\'{o}n problem.

\begin{theorem}
\label{thm_main}
Let $(M,g)$ be a Lorentzian manifold of the form \eqref{cylinder}--\eqref{splitting}. For $j=1,2$, let $V_j \in C^{\infty}(M)$. Suppose that the assumptions (H1)--(H4) are satisfied. Let $\tilde g$ be a smooth Lorentzian metric on $M$ that lies in a sufficiently small $C^2(M)$-neighborhood of $g$. Let $\Lambda_{V_j}$, $j=1,2$, be defined as in \eqref{DNmap} corresponding to the wave equation \eqref{pf0} on $(M,\tilde g)$ with $V=V_j$. Suppose that
$$ \Lambda_{V_1}\,f= \Lambda_{V_2}\,f,\quad \forall\, f \in H^1_0(\Sigma).$$
Then $V_1=V_2$ on $M$.
\end{theorem}

As an immediate corollary of the above result, we can solve the Lorentzian Calder\'{o}n problem near the Minkowski geometry. Recall that the Minkowski metric $\eta$ on $\R^{1+n}$ is defined via $\eta= -(dt)^2+(dx^1)^2+\ldots+(dx^n)^2.$
\begin{corollary}
	\label{thm_minkowski}
	Let $\Omega\subset \R^n$ be a compact connected set with non-empty interior and a smooth strictly convex boundary. For $j=1,2,$ let $V_j \in C^{\infty}_c(\R\times \Omega)$. Let $T>0$ be sufficiently large, let $\eta$ be the Minkowski metric and suppose that $g$ is a smooth metric on $M=[-T,T]\times \Omega$ that is in a sufficiently small $C^2(M)$-neighborhood of $\eta$. Let $\Lambda_{V_j}$ be defined as in \eqref{DNmap}, corresponding to the wave equation \eqref{pf0} on $(M,g)$ with $V=V_j$. If
	$$\Lambda_{V_1}\,f= \Lambda_{V_2}\,f,\quad \forall\, f \in H^1_0(\Sigma),$$
	then $V_1=V_2$ on $M$.
\end{corollary}

The key tool in proving Theorem~\ref{thm_main} is a new optimal unique continuation property for the wave equation in exterior regions of double null cones. Precisely, we prove the following unique continuation theorem.

\begin{theorem}\label{thm_uc}
Let $(M,g)$ be a Lorentzian manifold of the form \eqref{cylinder}--\eqref{splitting} and assume that (H1)--(H2) are satisfied. Let $\tilde g$ be a smooth Lorentzian metric on $M$ that lies in a sufficiently small $C^2(M)$-neighborhood of $g$. Let $X$ be a first order linear differential operator with smooth coefficients on $M$. Let $p \in M^{\text{int}}$ be such that $\mathscr E_p\cap \p M\subset \Sigma$, where $\mathscr E_p$ is defined by \eqref{exterior_null} associated to the manifold $(M,\tilde g)$ and $\Sigma=(-T,T)\times \p M_0$. Let $u \in H^{-s}(M)$ for some $s \geq 0$ be a distributional solution to 
    \begin{align*}
\Box u+X u=0 \quad  \text{on $\mathscr E_p$},
    \end{align*}
    where $\Box$ is the wave operator associated to $(M,\tilde g)$. Suppose that the traces $u$ and $\p_\nu u$ both vanish on the set $\Sigma \cap  \mathscr E_p$. Then, $u = 0$ on $\mathscr E_p$.
\end{theorem}

Theorem~\ref{thm_main} follows immediately from combining the above unique continuation theorem with the controllability arguments in our earlier work, see \cite[Section 6--7]{AFO}.

\subsection{Previous literature}
\label{previous_lit}

Before reviewing the literature of the Lorentzian Calder\'{o}n problem, let us make a comparison with related results in the elliptic setting. Recall that the analogous injectivity question $V\mapsto \Lambda_V$ with $\Lambda_V$ denoting the Dirichlet-to-Neumann map associated to a Riemannian manifold $(M,g)$ and with $\Box$ replaced by the Laplace--Beltrami operator on $(M,g)$ is the well known Calder\'{o}n problem. In this elliptic setting, the seminal work \cite{SU} proves uniqueness for the coefficient $V$ on Euclidean domains of dimension larger than two. Uniqueness in the two dimensional case was proved in the work \cite{Nach} for certain classes of smooth $V$ and later in \cite{Bukh} for general smooth $V$. We mention also that uniqueness is known when the manifold $(M,g)$ and $V$ are both real-analytic \cite{LU} . Outside of these categories, uniqueness is only known for certain manifolds $(M,g)$ with a Euclidean direction under additional assumptions \cite{DKSU,DKLS}. Here, existence of a Euclidean direction essentially means that the components of the Riemannian metric must be independent of one of the coordinates on the manifold. For a review of the literature of results on the (Riemannian) Calder\'{o}n problem we refer the reader to the survey article \cite{Uh}.

The results that are related to recovery of lower order coefficients in wave equations can in general be divided into two categories of time-independent
and time-dependent coefficients. Starting with the seminal work \cite{Bel87}, there is a rich literature of results that is related to the recovery of time-independent coefficients based on the Boundary Control (BC) method. The BC method fundamentally relies on the optimal unique continuation theorem of Tataru \cite{Tataru} (see also 
the important precursor by Robbiano \cite{Rob}, and related later results by Robbiano and Zuilly \cite{RZ} and Tataru 
\cite{TataruII}). This result says, in the case of the wave equation with analytic in time coefficients, that unique continuation principle holds across any non-characteristic hypersurface. Nonetheless, the unique continuation principle fails when the coefficients in the equation are only smooth \cite{Al}. 
In line with this, the works by Eskin \cite{E1,E2} solve the inverse problem of recovering time-analytic coefficients for the wave equation. We refer the reader to \cite{Bel,KKL} for a review of the results that are based on the BC method.

Outside the category of wave equations with time-analytic coefficients, results are scarce. We mention the works \cite{RS,Sal,St} that solve the problem of recovering time-dependent lower order coefficients in the Minkowski spacetime. The approach of Stefanov in \cite{St} uses the principle of propagation of singularities for the wave equation to reduce the inverse problem to the study of the injectivity of the light ray transform of the unkown coefficient. The inversion of this transform follows from Fourier analysis in the Minkowski geometry. The reduction step from injectivity of the map $V \mapsto \Lambda_V$ to injectivity of the light ray transform has been generalized to a broad geometric setting \cite{SY} but there are few results about injectivity of the light ray transform. Indeed, this transform is known to be injective only in the case of ultrastatic metrics \cite{FIKO} (see also the related earlier work \cite{KO19}), stationary metrics \cite{FIO}, and in the case of real-analytic spacetimes \cite{Stefanov1}, under certain additional convexity conditions.

In the recent work by the authors \cite{AFO}, an optimal unique continuation theorem was proved for the wave equation in Lorentzian geometries that satisfy certain two-sided curvature bounds first introduced by Andersson and Howard in \cite{AH98}. As a consequence of this unique continuation principle, together with some controllability theory for the wave equation in rough Sobolev spaces, we were able to show unique determination of zeroth order coefficients in certain geometries where no real analytic features are present.
The result \cite{AFO} covers perturbations of ultrastatic manifolds with strictly negative spatial sectional curvature. 
However, the Minkowski spacetime is on the boundary of the spacetimes allowed in \cite{AFO}, and its perturbations are not covered by the theory. Our proof of the unique continuation theorem in \cite{AFO} relied on the notion of spacetime convex functions that can be constructed under the two-sided curvature bounds \cite{AB08,AH98}, together with Carleman estimates with degenerate weights. In relation to the latter idea, we mention the earlier work \cite{Alexakis,Shao} that prove a similar type of unique continuation result in the Minkowski geometry.

In the present work, we have identified a weaker curvature bound (H1) that replaces the curvature bound in the previous work. This improvement is essentially related to a new comparison result for a Riccati equation compared to the comparison result that appears in the work of Alexander and Bishop, see \cite[Theorem 4.3]{AB08}. As a result of assuming only this weaker curvature bound, we also solve the Lorentzian Calder\'{o}n problem for spacetime perturbations of the Minkowski geometry. Let us mention also that the proof of the unique continuation principle in this paper is different from the previous work. For each point $p$ on $M$, we construct a function with strictly pseudo-convex level sets that give a foliation of the exterior of the double null cone $\mathscr E_p$. The level sets of the spacetime convex functions used in \cite{AFO} are pseudo-convex only in a non-strict manner.  

While finishing this article, we became aware of an upcoming preprint by Vaibhav Kumar Jena and Arick Shao, where they independently prove a unique continuation theorem in the near Minkowski setting. Their assumptions and proof are different from ours, however, Corollary~\ref{thm_minkowski} can be also proven by using their unique continuation result together with our previous work  \cite{AFO}. We also mention that, contrary to our unique continuation principle, their result comes with a Carleman estimate. To acknowledge simultaneity and independence of both the works, we agreed to post our respective preprints to arXiv at the same time.

\medskip\paragraph{\bf Acknowledgements}
S.A. acknowledges support from NSERC grant 488916. A.F. was supported by the Fields institute for research in mathematical sciences.

\section{Comparison result for a Riccati equation on $\R^n$}

Let us consider the Minkowski inner product on $\R^{n}$ with $n\geq 2$, that is defined by 
$$ \langle v,w\rangle = -v_0w_0+ \sum_{j=1}^{n-1} v_jw_j.$$ 
A vector $v\in \R^{n}$ is {\em null} if $\langle v,v\rangle=0$, 
and a matrix $L \in \R^{n \times n}$ is {\em symmetric} (with respect to the Minkowski metric) if 
    \begin{align*}
\langle L v,w\rangle = \langle v, Lw\rangle
    \end{align*}
for all $v,w \in \R^n$.
We have the following definition.
\begin{definition}
\label{def_null_neg}
A symmetric matrix $L\in \R^{n\times n}$ is null negative-definite, shortly $L\ND 0$, if 
    \begin{align*}
\langle Lv,v\rangle< 0, \quad \text{for all nonzero null vectors $v \in \R^{n}$},
    \end{align*}
 We also say that $L$ is null negative semi-definite, shortly $L\NS 0$, if 
    \begin{align*}
\langle Lv,v\rangle \leq 0, \quad \text{for all null vectors $v \in \R^{n}$}.
    \end{align*}
    \end{definition}
    
 We start with a lemma.
 
 \begin{lemma}
 \label{lem_null}
 Let $L\in \R^{n\times n}$, be a symmetric matrix that is null negative semi-definite but that it is not null negative-definite. Then there exists a nonzero null vector $x\in \R^n$  such that $Lx= \lambda x$ for some $\lambda \in \R$.
 \end{lemma}   
 \begin{proof}
 We begin by defining $e_0=(1,0,\ldots,0)$. Since $L$ is null negative semi-definite but not null negative-definite, there must exist a null vector $x=e_0 +\tilde{x}$, with $\langle e_0,\tilde{x}\rangle=0$ and $\langle{\tilde x,\tilde x}\rangle=1$ such that
 $$ \langle{Lx,x}\rangle=0.$$ 
When $n=2$, there are $c_0, c_1 \in \R$ such that 
    \begin{align}\label{Lx_lin_comb}
Lx = c_0 e_0 + c_1 \tilde x.
    \end{align}
The claim follows since $c_0 = c_1$ due to
    \begin{align*}
0 = \langle{Lx,x}\rangle= -c_0 + c_1.
    \end{align*}

When $n\geq3$, we proceed by letting $y$ to be any vector that satisfies
 \begin{equation}
 \label{y_cond} 
 \langle{y,y}\rangle=1,\quad\text{and}\quad \langle{y,e_0}\rangle=\langle{y,\tilde{x}}\rangle=0,
 \end{equation}
 and subsequently define for any $\epsilon \in (-1,1)$,
 $$z_\epsilon= e_0+\sqrt{1-\epsilon^2} \,\tilde x+ \epsilon y.$$
Observe that $z_\epsilon$ is null. Moreover, as $\langle{Lx,x}\rangle=0$, there holds 
$$ 0\geq\langle{Lz_\epsilon,z_\epsilon}\rangle = 2\epsilon \langle{Lx,y}\rangle + O(\epsilon^2).$$
Since the latter expression must be valid for all $y$ that satisfy \eqref{y_cond}, we deduce that (\ref{Lx_lin_comb}) holds for some $c_0, c_1 \in \R$, and conclude as in the case $n=2$.
 \end{proof}

 \begin{proposition}
 \label{prop}
Let $T>0$. Suppose that $B\in C([0,T];\R^{n \times n})$ satisfies
$$B(t) \ND 0\quad \text{for all $t\in[0,T]$}.$$ 
Let $L$ be the unique symmetric matrix that solves the Riccati equation 
\begin{equation}
\label{riccati} 
tL'(t) +L^2(t) -L(t) +t^2 B(t) =0 \quad \text{for all $t\in (0,T]$},
\end{equation}
with $L(0)=\textrm{id}$, where $\textrm{id}$ stands for the identity matrix. Then, 
$$-L(t) \ND 0\quad \text{for all $t \in (0,T]$}.$$
 \end{proposition}

\begin{proof}
By differentiating \eqref{riccati} we obtain that
$$ L'+tL'' +L'L+LL' -L' +(t^2B)'=0.$$
Plugging $t=0$ in the latter expression, we deduce that
\begin{equation}
\label{S_0}
L'(0)=0.
\end{equation}
Differentiating the equation again and plugging $t=0$, we also obtain
\begin{equation}
\label{S_1}
L''(0)=-\frac{2}{3}B(0).
\end{equation}
Together with the fact that $B(0) \ND 0$, it follows that 
\begin{equation}
\label{local_est}
 -L(t) \ND 0, \quad \text{for all $t \in (0,\delta)$,}\end{equation}
for some small $\delta>0$. To show that $-L \ND 0$ on the entire interval $(0,T]$, we assume that this is not the case and reach a contradiction. But then, by \eqref{local_est}, there is $t_0\in (0,T)$ such that
$-L(t)$ is null negative-definite for all $t \in (0,t_0)$, and that $-L(t_0)$ is null negative semi-definite but not null negative-definite. Thus, in view of Lemma~\ref{lem_null}, we have
\begin{equation}
\label{null_prop}
L(t_0)x_0 = \lambda x_0,
\end{equation}
for some nonzero null vector $x_0$ and some $\lambda\in \R$. Next, let us define
$$ f(t) = -\langle{L(t)x_0,x_0}\rangle.$$
Then applying \eqref{riccati} we may write
\begin{equation}
\label{derivative_iden}
\begin{aligned}
t_0f'(t_0)&= \langle{-t_0L'(t_0)x_0,x_0}\rangle\\
&=\langle{L(t_0)x_0,L(t_0)x_0}\rangle-\langle{L(t_0)x_0,x_0}\rangle+t_0^2 \langle{B(t_0)x_0,x_0}\rangle.
\end{aligned}
\end{equation}
Note that since $x_0$ is null and \eqref{null_prop} holds, we have
$$\langle{L(t_0)x_0,L(t_0) x_0}\rangle=0\quad \text{and}\quad \langle{L(t_0)x_0,x_0}\rangle=0.$$
Thus, we may simplify \eqref{derivative_iden} to obtain
$$
t_0f'(t_0)=t_0^2 \langle{B(t_0)x_0,x_0}\rangle<0.
$$
where we have also used the fact that $B\ND 0$. 
On the other hand, $f < 0$ on $(0,t_0)$, since $-L \ND 0$ on this interval, and also $f(t_0) = 0$.
Hence $t_0f'(t_0)\geq 0$, a contradiction. 
\end{proof}

\section{Proof of Theorem~\ref{thm_uc} via a strictly pseudo-convex foliation}
\label{uc_sec}
The aim of this section is to prove the unique continuation property claimed in Theorem~\ref{thm_uc}. We define a stronger variant of (H1) with a strict inequality, namely,
\begin{itemize}
	\item[$(\textrm{H1})^\prime$]{For any point $p\in M$, any spacelike vector $v\in T_pM$, and any nonzero null vector $N\in T_pM$ with $g(v,N)=0$, there holds $$g(R(N,v)v,N)<0.$$}
\end{itemize}

\begin{lemma}
	\label{lem_conf_strict}
	Let $(M,g)$ be a Lorentzian manifold of the form \eqref{cylinder}--\eqref{splitting} that satisfies (H1)--(H2). There exists $f\in C^{\infty}(M)$ such that the Lorentzian manifold $(M,e^{2f}g)$ satisfies $(\textrm{H1})^\prime$--(H2) with $(M,g)$ replaced by $(M,e^{2f}g)$.  
\end{lemma}

\begin{proof}
	Let $|\cdot|$ be the length with respect to an auxiliary Riemannian metric on $M$. There is a constant $C>0$ such that given any $p\in M$ and any null vector $N\in T_pM$ there holds
	\begin{equation}
	\label{N_est}
	|N|\leq C |N(t)|.
	\end{equation} 
	This estimate can first be proved locally near a point $p$ in $M$ by using the normal coordinates centered at $p$ and subsequently by compactness of $M$ we can derive the estimate globally. It follows from \eqref{N_est} that there exists some constant $C_0>0$ only depending on $(M,g)$ such that given any $p\in M$ and any null vector $N\in T_pM$, there holds
	$$|\Hess t(N,N)| \leq C_0 N(t)^2.$$
	Fixing any $\lambda>2C_0$ and defining $\tau(t,x)=e^{\lambda t}$ for all $(t,x)\in M$, it follows that
	\begin{equation}
	\label{Hesstau}
	\Hess \tau(N,N)= e^{\lambda t}(\lambda \Hess t(N,N)+\lambda^2 N(t)^2)\geq C_1 N(\tau)^2,
	\end{equation}
	for some positive constant $C_1$ depending on $C_0$ and $\lambda$. Next, let us define $\tilde{g}=e^{2f}g$ with $f=\delta\,\tau$. We choose $\delta \in (0,C_1)$ to be sufficiently small so that $(M,\tilde g)$ satisfies (H2) with $(M,g)$ replaced by $(M,\tilde g)$. This is always possible as (H2) is an open condition with respect to small perturbations of the metric, see \cite[Section 3]{AFO} for the proof. We claim that $(M,\tilde g)$ also satisfies $(\textrm{H1})^\prime$ with $(M,g)$ replaced by $(M,\tilde g)$. To see this, we write $\widetilde R$ for the curvature tensor on $(M,\tilde g)$ and observe that
	\begin{equation}
	\label{R_conf_0} 
	\widetilde R= e^{2f}R - e^{2f}\, g\KN \left(\Hess f -df\otimes df + \frac{1}{2}|df|^2g\right),
	\end{equation}
	where $\KN$ is the Kulkarni--Nomizu product. We need to prove that there holds
	$$ \widetilde{R}(N,v,N,v)=\tilde g(\widetilde R(N,v)v,N) <0,$$
	for all $p \in M$, all $v\in T_pM$ and all nonzero $N\in T_pM$ that satisfy 
	$$ \tilde g(v,v)=1, \quad \tilde g(N,N)=0, \quad \text{and}\quad \tilde g(N,v)=0.$$
	Using equation \eqref{R_conf_0} together with the bound \eqref{Hesstau}, we have
	\begin{multline} 
	\label{eq_curv_0}
	e^{-2f}\widetilde R(N,v,N,v) = R(N,v,N,v) - \Hess f(N,N) +(df(N))^2=\\
	R(N,v,N,v)-\delta\Hess \tau(N,N)+ \delta^2(N(\tau))^2<0,
	\end{multline}
	as $\delta\in (0,C_1)$ and $R(N,v,N,v)\leq 0$, thus completing the claim.	
\end{proof}

\begin{lemma}
	\label{lem_conf_pert_0}
	Let $(M,g)$ be a Lorentzian manifold of the form \eqref{cylinder}--\eqref{splitting} that satisfies $(\textrm{H1})^\prime$--(H2). Given any smooth Lorentzian metric $\tilde g$ in a sufficiently small $C^2(M)$-neighborhood of $g$, the manifold $(M,\tilde g)$ also satisfies $(\textrm{H1})^\prime$--(H2) with $(M,g)$ replaced by $(M,\tilde g)$.
\end{lemma}

\begin{proof}
The fact that (H2) remains valid on manifolds $(M,\tilde g)$, with $\tilde g$ a small perturbation of $g$, follows from \cite[Section 3]{AFO}. We write $\epsilon>0$ for the distance in $C^2(M)$ of the metric $\tilde g$ from $g$ and write $\widetilde R$ for the curvature tensor on $(M,\tilde g)$.

Let $|\cdot|$ be the length with respect to an auxiliary Riemannian metric on $M$. Observe that since $(M,g)$ satisfies $(\textrm{H1})^\prime$, it follows that there exists a constant $\kappa<0$ such that given any $p\in M$, any  vector $N\in T_pM$ with $g(N,N)=0$ and any spacelike $v\in T_pM$ with $g(v,N)=0$, there holds
    \begin{align}\label{Rkappa}
R(N,v,N,v) \leq \kappa\, |N|^2 |v|^2.
    \end{align} 
Also, there is $\delta > 0$ such that $g(v,v) > \delta$
for all $p \in M$ and for all $v \in T_p M$ satisfying $|v|=1$ and $g(v,v) > 0$. Hence for small $\epsilon > 0$ there holds 
$\tilde g(v,v) > \delta/2$ for all $p \in M$ and for all $v \in T_p M$ satisfying $|v|=1$ and $\tilde g(v,v) > 0$.

Our goal is to show that given any $\epsilon>0$ sufficiently small, any $p\in M$, and any vectors $\tilde N\in T_pM\setminus 0$ and $\tilde v\in T_pM$ satisfying 
\HOX{No need to normalize $\tilde v$ here}
	\begin{equation}
	\label{N_v_cond}
	\tilde g(\tilde N,\tilde N)=0, \quad \tilde g(\tilde v, \tilde v) > 0, \quad \text{and}\quad \tilde g(\tilde v,\tilde N)=0,
	\end{equation}
	there holds
	\begin{equation}
	\label{R_tilde_neg}
	\widetilde R(\tilde N,\tilde v,\tilde N,\tilde v)<0.
	\end{equation}
It is enough to consider vectors $\tilde N$ and $\tilde v$ that are normalized by $|\tilde N| = |\tilde v| = 1$.
 
Let us show that there is $a \in \R$ of size $\mathcal O(\epsilon)$ such that $N = a \p_t + \tilde N$ satisfies $g(N,N) = 0$.
Solving for $a$ from this equation gives
    \begin{align*}
a = \frac{g(\p_t, \tilde N) \pm \sqrt{g(\p_t, \tilde N)^2 - 4 g(\p_t, \tilde N) g(\tilde N, \tilde N)}}{2 g(\p_t, \p_t)}.
    \end{align*}
We choose the sign that is oposite to the sign of $g(\p_t, \tilde N)$, and $a = \mathcal O(\epsilon)$ follows then from 
    \begin{align*}
g(\tilde N, \tilde N) = \tilde g(\tilde N, \tilde N) + \mathcal O(\epsilon) = \mathcal O(\epsilon).
    \end{align*}
A similar argument shows that there is $b \in \R$ of size $\mathcal O(\epsilon)$ such that $v = b \p_t + \tilde v$ satisfies $g(v,N) = 0$.
Moreover, for small enough $\epsilon > 0$,
    \begin{align*}
g(v,v) = \tilde g(\tilde v, \tilde v) + \mathcal O(\epsilon)
> \delta/2 + \mathcal O(\epsilon) > 0.
    \end{align*}
Now (\ref{Rkappa}) gives for small enough $\epsilon > 0$
	\begin{align*}
\widetilde R(\tilde N,\tilde v,\tilde N,\tilde v) 
= 
R(N,v,N,v) + \mathcal O(\epsilon) 
\leq 
\kappa |N|^2 |v|^2 +\mathcal O(\epsilon) 
= \kappa +\mathcal O(\epsilon) < 0.
	 \end{align*}
\end{proof}

Combining the latter two lemmas, we obtain the following corollary.

\begin{corollary}
	\label{cor_conf_pert}
	Let $(M,g)$ be a Lorentzian manifold of the form \eqref{cylinder}--\eqref{splitting} that satisfies (H1)--(H2). There exists $f\in C^{\infty}(M)$ such that given any smooth Lorentzian metric $\tilde g$ on $M$ that lies in a sufficiently small $C^2(M)$-neighborhood of $g$, the manifold $(M,e^{2f}\tilde g)$ satisfies $(\textrm{H1})^\prime$--(H2) with $(M,g)$ replaced by $(M,e^{2f}\tilde g)$.  
\end{corollary}

\begin{remark}
	\label{remark_conf_curv}
	Recall that Theorem~\ref{thm_uc} claims a unique continuation property stated on manifolds $(M,\tilde g)$ that as in the hypothesis of the latter corollary lie in a sufficiently small $C^2(M)$-neighborhood of manifolds $(M,g)$ that satisfy (H1)--(H2). Owing to the transformation rule for the wave operator under conformal scalings of the metric (see
	for example \cite{LO21}), it is straightforward to see that Theorem~\ref{thm_uc} follows from Corollary~\ref{cor_conf_pert} after we prove an analogous unique continuation theorem on manifolds $(M,g)$ that satisfy $(\textrm{H1})^\prime$--(H2).
\end{remark}

In view of the latter remark, henceforth we will assume without any loss of generality that $(M,g)$ is a manifold of the form \eqref{cylinder}--\eqref{splitting} that satisfies $(\textrm{H1})^\prime$--(H2). Given any $p \in M^{\textrm{int}}$ we let $e_0,\dots,e_n$ be an orthonormal basis of $T_p M$
in the sense of \cite[Lemma 24, p. 50]{O}, that is,
for distinct $j,k = 0,\dots,n$,
\begin{align*}
g(e_j, e_k)= 0, \quad
g(e_j, e_j) = \epsilon_j,
\end{align*}
where $\epsilon_0 =-1$, and $\epsilon_j=1$ for $j=1,2,\ldots,n$. We define the Lorentzian distance function $r_p$ on $\mathscr E_p$ by the expression
\begin{align}\label{def_r}
r_p(y) = \left( -(y^0)^2 +(y^1)^2+ \dots +(y^n)^2 \right)^{1/2}
\end{align}
in the normal coordinate system $y = (y^0, \dots, y^n) \mapsto \exp_p(y^j e_j)$.

We remark that (H2) implies that $r_p \in C^{\infty}(\mathscr E_p)$. We are ready to state a proposition that  ensures that the level sets of $r_p$ are strictly pseudo-convex. 

\begin{proposition}
\label{th_comp}
Let $(M,g)$ be a Lorentzian manifold of the form \eqref{cylinder}--\eqref{splitting} that satisfies $(\textrm{H1})^\prime$--(H2). Let $p \in M^{\textrm{int}}$ and let $r_p$ be as \eqref{def_r}. Given any point $q \in \mathscr E_p$, there holds 
\begin{equation}
\label{pseudo_convex_eq}
\textrm{Hess}\, r^2_p(N,N)>0, 
\end{equation}
for any nonzero $N\in T_qM$ with $\pair{N,\nabla r_p}=\pair{N,N}=0.$
\end{proposition}

Observe now that our unique continuation principle, Theorem~\ref{thm_uc}, follows immediately from combining Proposition~\ref{th_comp} together with \cite[Lemma 5.2]{AFO}, \cite[Theorem 28.3.4]{Ho4} and Remark~\ref{remark_conf_curv}. The rest of this section is devoted to the proof of Proposition~\ref{th_comp}. 
	
\subsection{Radial curvature equation}

The directional curvature operator is defined by
  \begin{align*}
R_v w = R(w,v) v, \quad v,w \in T_x M,\ x \in M,
    \end{align*}
where $R$ is the $(1,3)$ curvature tensor, that is,
for vector fields $X$, $Y$ and $Z$,
    \begin{align*}
R(X,Y)Z = \nabla_X \nabla_Y Z - \nabla_Y \nabla_X Z - \nabla_{[X,Y]} Z.
    \end{align*}
Here $\nabla$ is the covariant derivative. 
We also recall that a function $r : M \to \R$ is called a {\em distance function} if $\pair{dr, dr} = 1$. Writing $\p_r$ for the gradient of $r$ it is straightforward to see that
$$\nabla_{\p_r} \p_r = 0.$$

\begin{lemma}
\label{lem_S}
Let $r : M \to \R$ be a distance function,
and consider the shape operator $S$ corresponding to $r^2/2$,
defined by
    \begin{align*}
\frac12 \Hess r^2(v, w) = \pair{Sv, w},
    \end{align*}
    for all $v,w \in T_xM$ with $\pair{v,\p_r}=\pair{w,\p_r}=0$.
In other words, $SX = \nabla_X Y$, for a vector field $X$ with $\pair{X,Y}=0$, where $Y$ is the gradient of $r^2/2$.
Then $S$ satisfies the radial curvature equation
    \begin{align}\label{radial_curvature}
\nabla_{Y} S -S + S^2 + R_Y = 0,
    \end{align}
    on $Y^{\perp}=\{X\in T_xM\,:\, \pair{X,Y}=0\}$.
\end{lemma}
\begin{proof}
First, let us show that $S$ and $R_Y$ are closed on $Y^\perp$. Observe that $Y = r \p_r$. Hence given any $X \in Y^\perp$, there holds
$$ \pair{SX,Y}= \pair{\nabla_XY,Y} = \frac{1}{2} X(\pair{Y,Y})= \frac12 X(r^2) = rX(r) = \pair{X,Y} =0,$$
thus proving that $SX \in Y^\perp$. Note also that
$$ \pair{R_Y X,Y}=\pair{R(X,Y)Y,Y}=0,$$
by anti-symmetry of $R$ in its last two indices. Thus, $R_YX \in Y^\perp$ as well.
Next, we write
    \begin{align*}
(\nabla_{Y} S)X + S^2 X 
&= 
\nabla_{Y}(SX) - S\nabla_{Y}X + S^2 X 
\\&= 
\nabla_{Y}\nabla_X Y - \nabla_{\nabla_{Y}X} Y 
+ \nabla_{\nabla_X Y} Y
= \nabla_Y \nabla_X Y + \nabla_{[X,Y]} Y.
    \end{align*}
On the other hand,
    \begin{align*}
R_{Y}X = R(X,Y)Y 
= \nabla_X \nabla_Y Y - \nabla_Y \nabla_X Y - \nabla_{[X,Y]}Y.
    \end{align*}
Moreover, there holds 
    \begin{align*}
\nabla_Y Y = r \nabla_{\p_r}(r \p_r)
= r \p_r + r^2 \nabla_{\p_r} \p_r = Y.
    \end{align*}
Thus $\nabla_X \nabla_Y Y = SX$
and
    \begin{align*}
(\nabla_{Y} S)X + S^2 X 
= -R_Y X + SX.
    \end{align*}
\end{proof}

\subsection{The comparison result on $(M,g)$}

Let us begin by showing that $r_p$, as defined in \eqref{def_r}, is a distance function on the subset $\mathscr E_p$.
Consider the local hyperquadric 
    \begin{align*}
Q = \{\omega \in T_p M : \pair{\omega,\omega} = 1\}.
    \end{align*}
In the region $\mathscr E_p$, we consider the polar coordinates $y = r\omega$ with $r>0$ and $\omega \in Q$, 
where $y=(y^0, \dots, y^n)$ is as in \eqref{def_r}.
The Gauss lemma, see e.g. \cite[Lemma 1, p. 127]{O}, implies that in these coordinates, with $r$ having the index $n$,
the metric tensor has the form
    \begin{align*}
\begin{pmatrix}
h(r\omega) & 0
\\
0 & 1
\end{pmatrix}.
    \end{align*}
As $r_p(y) = r$ for $y = r\omega$, 
it follows that $\pair{dr_p, dr_p} = 1$.

\begin{proof}[Proof of Proposition~\ref{th_comp}]
We will prove inequality \eqref{pseudo_convex_eq} at an arbitrary point $q \in \mathscr E_p$, by using the Riccati equation \eqref{radial_curvature} associated to the distance function $r_p$ along the radial geodesic segment $\gamma$ that connects the point $p$ to $q$. We write $p=\gamma(0)$ and consider an orthonormal frame $\{\tilde{e}_j(0)\}_{j=0}^{n-1}$ on $\dot{\gamma}(0)^\perp$ in the sense that
for distinct $j,k = 0,\dots,n-1$,
    \begin{align*}
\pair{\tilde{e}_j, \tilde{e}_k} = 0, \quad
\pair{\tilde{e}_j, \tilde{e}_j} = \epsilon_j,
    \end{align*}
where $\epsilon_0 =-1$, and $\epsilon_j=1$ for $j=1,2,\ldots,n-1$. For each $j=0,1,\ldots,n-1$, we define $\tilde{e}_j(s) \in \dot{\gamma}(s)^\perp$ to be the parallel transport of the vector $\tilde{e}_j(0)$ along $\gamma$ from $\gamma(0)$ to $\gamma(s)$. Note that for all $v \in \dot{\gamma}^\perp$,
    \begin{align*}
v = \sum_{j=0}^{n-1} \epsilon_j\, \pair{v, \tilde{e}_j}\, \tilde{e}_j.
    \end{align*}
In particular, 
the matrix $\tilde S$ of a linear map $S$ on $\dot \gamma^\perp$,
defined by $S\tilde{e}_k = \tilde S^j_k \tilde{e}_j$, satisfies
$\tilde S^j_k = \epsilon_j \pair{S \tilde{e}_k, \tilde{e}_j}$. Let us now take $S$ as in Lemma~\ref{lem_S} restricted to the geodesic segment $\gamma$ and use the abbreviated notation $r$ in place of $r_p$. Note that the radial curvature equation (\ref{radial_curvature}) implies on $\gamma$
    \begin{multline*}
r\p_r \epsilon_j \pair{S\tilde{e}_j, \tilde{e}_k} = \epsilon_j \pair{(\nabla_{r\p_r} S)\tilde{e}_j, \tilde{e}_k}\\
= \epsilon_j \pair{S\tilde{e}_j, \tilde{e}_k}
- \epsilon_j \pair{S^2 \tilde{e}_j, \tilde{e}_k}
- r^2 \epsilon_j \pair{R_{\p_r} \tilde{e}_j, \tilde{e}_k}.
    \end{multline*}
Thus the matrix $\tilde S^j_k$ of $S$
satisfies the Riccati equation 
    \begin{align}\label{Riccati}
r \tilde S' - \tilde S + \tilde S^2 + r^2 \tilde R = 0,
    \end{align}
with $\tilde R$ the matrix of $R_{\p_r}$ on $\dot{\gamma}^\perp$ with respect to the frame $\{\tilde e_j\}_{j=0}^{n-1}$ (note that in the flat case, with the metric tensor
$\sum_{j=0}^n\epsilon_j (dx^j)^2$,
there holds $\tilde S = \textrm{id}$). In general, $\tilde S(0) = \textrm{id}$.
Moreover, the curvature bound $(\textrm{H1})^\prime$ implies that the matrix $\tilde R$ is null negative definite in the sense of Definition~\ref{def_null_neg}. The proof is completed, thanks to Proposition~\ref{prop} with $L=\tilde S$ and $B=\tilde R$. 
\end{proof}

\ifoptionfinal{}{
}
\end{document}